\documentclass[11pt]{amsart}
\usepackage[T1]{fontenc}
\usepackage{lmodern}
\usepackage[utf8]{inputenc}
\usepackage{amsmath,amssymb,amsthm}
\usepackage{mathtools}
\usepackage{booktabs}
\usepackage[margin=1.1in]{geometry}
\usepackage{placeins}
\usepackage{pgfplots}
\pgfplotsset{compat=1.18}
\definecolor{copper}{RGB}{184,115,51}
\definecolor{lightblue}{RGB}{70,140,210}
\usepackage{hyperref}

\DeclareMathOperator{\rad}{rad}

\theoremstyle{plain}
\newtheorem{theorem}{Theorem}[section]
\newtheorem{proposition}[theorem]{Proposition}
\newtheorem{lemma}[theorem]{Lemma}
\newtheorem{corollary}[theorem]{Corollary}
\newtheorem{conjecture}[theorem]{Conjecture}

\theoremstyle{definition}
\newtheorem{definition}[theorem]{Definition}
\newtheorem{example}[theorem]{Example}

\theoremstyle{remark}

\numberwithin{equation}{section}

\begin{document}

\title{Pasten's arithmetic derivatives and the decomposition into gains}

\author{R. Laniewski}
\author{K. M\"uller}

\date{October 19, 2026}



\begin{abstract}\label{abs:pasten}
Pasten attached to each coprime triple $a+b=c$ a lattice of arithmetic derivations that obey the Leibniz rule and are additive on this one equation. An arithmetic Wronskian then gives an $abc$-type inequality governed by the norm $\|\psi\|$ of the derivation, and Pasten showed that sufficiently small derivations exist under the $abc$ conjecture and give an $abc$ inequality with a fixed exponent.

We determine exactly which values this Wronskian takes, and we show that its estimate with exact prime weights has the best possible coefficient. The quality $X=\log c/\log\operatorname{rad}(abc)$ factors into an approximation gain and a power gain for each chosen presentation of the triple, and the resulting inequality bounds $X$ and both gains. For coprime polynomials the same argument gives $G_a < 1$ and $G_p < 3$. We also show that the largest absolute value of the derivation on $a$, $b$ and $c$ exceeds $c/\operatorname{rad}(abc)$ whenever the derivation is additive on the triple with non-zero Wronskian, so that it gives no bound on the quality by itself, and we determine its minimum when $a=1$. For Fermat triples the presentation by bases has approximation gain greater than $n/3$, which recovers Pasten's conditional asymptotic Fermat theorem. Finally, for five triples we certify the minimal norms and find the inequality within $0.08$ of the quality, with equality for the triples $1+80=81$ and $1+2400=2401$.
\end{abstract}

\maketitle

\section{Introduction}\label{sec:intro}

The proof of the Mason--Stothers theorem for polynomials rests on the derivative~\cite{Mason1984,Stothers1981,Snyder2000}. Its arithmetic counterpart is the $abc$ conjecture of Masser and Oesterl\'e~\cite{Oesterle1988Fermat}, which asserts that for every $\varepsilon>0$ only finitely many coprime triples $a+b=c$ satisfy $c>\rad(abc)^{1+\varepsilon}$. Over $\mathbb{Z}$ there is no derivative of the same kind, since every additive map that obeys the Leibniz rule vanishes. Following an observation of Pheidas that additivity is needed only finitely many times, Pasten constructed, for a given coprime triple $a+b=c$, derivatives that obey the Leibniz rule and are additive on this one equation~\cite{Pasten2022}. Their size is measured by a norm $\|\psi\|$, the largest absolute value taken on the primes. Pasten proved an $abc$-type inequality in which $\|\psi\|$ appears, showed by geometry of numbers that such derivatives exist with controlled norm, and proved that sufficiently small derivatives exist under the $abc$ conjecture and give an $abc$ inequality with a fixed exponent.

The present note places Pasten's inequality in the quality plane of Laniewski--M{\"u}ller~\cite{LaniewskiMuller2026}. In this plane the quality $X$ and the weak quality $Y$ of a triple serve as coordinates, and a chosen presentation $P$ of the triple factors the quality into an approximation gain and a power gain, following M\"uller and Taktikos~\cite{MullerTaktikos2026} and de Weger~\cite{deWeger2026}. Section~\ref{sec:plane} recalls these notions. Section~\ref{sec:polynomial} treats polynomials, where the Wronskian proof of the Mason--Stothers theorem gives $G_a<1$ and $G_p<3$ for every presentation. Section~\ref{sec:pasten} recalls Pasten's construction, determines the exact image and the optimal norm coefficient of the Wronskian, and states the inequality as a bound on $X$, from which the bounds on the two gains follow for every presentation. Under Pasten's small derivatives conjecture with exponent $\eta$, the quality is then asymptotically at most $1/(1-\eta)$.

Section~\ref{sec:norms} compares Pasten's norm with the largest of $|d^{\psi}a|$, $|d^{\psi}b|$ and $|d^{\psi}c|$. This maximum exceeds $c/\rad(abc)$ for every derivation that is additive on the triple and has non-zero Wronskian, and for $a=1$ its minimum is determined by divisibility. Section~\ref{sec:fermat} shows that for a Fermat triple the presentation by bases has approximation gain greater than $n/3$, which recovers the bound $n\le 3/(1-\eta)$ outside finitely many Fermat triples, as in Pasten's conditional form of the asymptotic Fermat theorem. Section~\ref{sec:data} gives minimal derivatives for five triples with certificates of minimality, and evaluates for them the inequality with exact divisibility and prime weights.

\section{The quality plane and presentations}\label{sec:plane}

Throughout, a \emph{triple} is a triple $(a,b,c)$ of coprime positive integers with $a+b=c$ and $a<b$, and we write $R=\rad(abc)$. Its \emph{quality}, \emph{weak quality} and \emph{balance} are
\[
X=\frac{\log c}{\log R},\qquad Y=\frac{\log(abc)}{\log R},\qquad \lambda=\frac{Y}{X}=\frac{\log(abc)}{\log c}.
\]
Since $a<c$ and $b<c$, one has $\lambda<3$ for every triple. If $a\ge 2$, then $ab\ge 2b>c$, whence $\lambda>2$~\cite[Lemma~2.2]{LaniewskiMuller2026}.

\begin{definition}\label{def:presentation}
A \emph{presentation} of a triple is a positive integer $P$ with $P\mid abc$ and $\rad(P)=R$. Its \emph{approximation gain}, \emph{power gain} and \emph{exponent factor} are
\[
G_a=\frac{\log c}{\log P},\qquad G_p=\frac{\log P}{\log R},\qquad f=\frac{\log(abc)}{\log P}.
\]
\end{definition}

Since $R=\rad(P)$ divides $P$, one has $R\le P\le abc$, and these quantities satisfy
\begin{equation}\label{eq:dictionary}
G_aG_p=X,\qquad fG_p=Y,\qquad 1\le G_p\le Y,\qquad G_a\le X
\end{equation}
for every presentation~\cite[Proposition~4.2]{LaniewskiMuller2026}. A bound on the quality therefore bounds the approximation gain of every presentation, and a bound on the weak quality bounds every power gain. The two extreme presentations are the radical $P=R$, for which $G_a=X$ and $G_p=1$, and the product $P=abc$, for which $G_p=Y$ and $f=1$.

\section{The polynomial model}\label{sec:polynomial}

Let $A$, $B$, $C$ be coprime polynomials in $\mathbb{C}[t]$ with $A+B=C$, not all constant. Put
\[
N=\max\{\deg A,\deg B,\deg C\},\qquad D=\deg\rad(ABC),\qquad T=\deg(ABC)
\]
where the radical $\rad(F)$ of a non-zero polynomial $F$ is the monic polynomial whose roots are the distinct roots of $F$, each with multiplicity one, so that $D$ is the number of distinct roots of $ABC$. The \emph{Wronskian} of the triple is $W=AB'-BA'$. From $A'+B'=C'$ one obtains
\begin{equation}\label{eq:poly-wronskian}
W=AB'-BA'=AC'-CA'=CB'-BC'.
\end{equation}

\begin{lemma}\label{lem:poly-divisibility}
The Wronskian $W$ is non-zero, and $ABC/\rad(ABC)$ divides $W$.
\end{lemma}

\begin{proof}
If $W=0$, then $(A/B)'=0$, so that $A$ is a constant multiple of $B$. Coprimality then forces $A$ and $B$ to be constant, and $C$ as well, which is excluded. A root of $A$ of multiplicity $m$ is a root of $A'$ of multiplicity $m-1$, hence a root of $AB'-BA'$ of multiplicity at least $m-1$. The same holds for the roots of $B$, and for the roots of $C$ through the form $AC'-CA'$ of~\eqref{eq:poly-wronskian}. Since $A$, $B$ and $C$ are coprime, the claim follows.
\end{proof}

\begin{theorem}[Mason--Stothers]\label{thm:mason}
One has $N\le D-1$.
\end{theorem}

\begin{proof}
Choose in~\eqref{eq:poly-wronskian} the form of $W$ that omits a polynomial of degree $N$. In this form each of the two products has degree at most the sum of the two remaining degrees minus one, with the convention $\deg 0=-\infty$ when one of the two remaining polynomials is constant, so that $\deg W\le T-N-1$. Lemma~\ref{lem:poly-divisibility} gives $T-D\le\deg W$, and the two inequalities combine to $N\le D-1$.
\end{proof}

\begin{definition}\label{def:poly-presentation}
A \emph{presentation} of the polynomial triple is a monic polynomial $P$ dividing $ABC$ with $\rad(P)=\rad(ABC)$. Its approximation gain and power gain are $G_a=N/\deg P$ and $G_p=\deg P/D$.
\end{definition}

Writing $H=\deg P$, one has $D\le H\le T$, and the gains satisfy $G_aG_p=X=N/D$, with weak quality $Y=T/D$.

\begin{corollary}\label{cor:poly-gains}
Every presentation of a polynomial triple satisfies $G_a<1$ and $G_p<3$.
\end{corollary}

\begin{proof}
By Theorem~\ref{thm:mason}, $G_a=N/H\le N/D<1$. Since each of $A$, $B$ and $C$ has degree at most $N$, one has $T\le 3N$, and therefore $G_p=H/D\le T/D\le 3N/D<3$.
\end{proof}

In the polynomial case the quality and the weak quality thus satisfy $X<1$ and $Y<3$, and these bounds pass to every presentation. Both come from the drop in degree under differentiation, which the inequality $\deg W\le T-N-1$ expresses.

\section{Pasten's arithmetic derivatives}\label{sec:pasten}

We recall the construction of~\cite{Pasten2022}. Let $\Omega$ be the free abelian group on symbols $\xi_p$, one for each prime $p$, and define $d\colon\mathbb{Z}\to\Omega$ by $d0=0$ and
\[
dn=n\sum_{p\mid n}\frac{v_p(n)}{p}\,\xi_p
\]
for $n\ne 0$. This map obeys the Leibniz rule. A \emph{derivation} is a homomorphism $\psi\colon\Omega\to\mathbb{Z}$ with finite norm $\|\psi\|=\sup_p|\psi(\xi_p)|$, and its \emph{arithmetic derivative} is $d^{\psi}=\psi\circ d\colon\mathbb{Z}\to\mathbb{Z}$. For a triple $(a,b,c)$, let $\mathcal{T}(a,b)$ be the group of derivations supported on the primes dividing $abc$ and satisfying $d^{\psi}a+d^{\psi}b=d^{\psi}c$. This single additivity condition is one homogeneous linear equation in the values $\psi(\xi_p)$, and $\mathcal{T}(a,b)$ has rank $\omega(abc)-1$~\cite[Lemma~2.4]{Pasten2022}.

The \emph{Wronskian} of $\psi$ is $W^{\psi}=a\,d^{\psi}b-b\,d^{\psi}a$, and $a$ and $b$ are \emph{$\psi$-independent} when $W^{\psi}\ne 0$. For $\psi\in\mathcal{T}(a,b)$, additivity gives, in analogy with~\eqref{eq:poly-wronskian},
\begin{equation}\label{eq:int-wronskian}
W^{\psi}=a\,d^{\psi}b-b\,d^{\psi}a=a\,d^{\psi}c-c\,d^{\psi}a=c\,d^{\psi}b-b\,d^{\psi}c
\end{equation}
and the definition of $d$ gives
\begin{equation}\label{eq:wronskian-sum}
W^{\psi}=ab\Bigl(\sum_{p\mid b}\frac{v_p(b)}{p}\psi(\xi_p)-\sum_{p\mid a}\frac{v_p(a)}{p}\psi(\xi_p)\Bigr).
\end{equation}

\begin{lemma}\label{lem:int-divisibility}
Put $g(1)=0$ and
\begin{equation}\label{eq:evaluation-gcd}
g(n)=\gcd_{p\mid n}\frac{n\,v_p(n)}{p}\qquad(n\ge 2).
\end{equation}
Write $g_a=g(a)$, $g_b=g(b)$ and $g_c=g(c)$, and put
\begin{equation}\label{eq:wronskian-step}
\Delta=\frac{\gcd(c\,g_ag_b,\ b\,g_ag_c,\ a\,g_bg_c)}{\gcd(g_a,g_b,g_c)}.
\end{equation}
Then the Wronskians of the derivations in $\mathcal{T}(a,b)$ form exactly the group
\begin{equation}\label{eq:wronskian-image}
\{W^{\psi}\colon\psi\in\mathcal{T}(a,b)\}=\Delta\mathbb{Z}
\end{equation}
and $abc/R$ divides $\Delta$. Here $\Delta$ is well defined, since $b\ge 2$ gives $g_b\ge 1$ and hence $\gcd(g_a,g_b,g_c)\ge 1$. Moreover, every positive integer $n$ satisfies $\sum_{p\mid n}v_p(n)/p\le\log n/(2\log 2)$.
\end{lemma}

\begin{proof}
Since $a$, $b$ and $c$ are pairwise coprime, the values of $\psi$ on the primes of each of them can be chosen independently, and by B\'ezout's identity the values of $d^{\psi}a$, $d^{\psi}b$ and $d^{\psi}c$ are $g_ax$, $g_by$ and $g_cz$ with arbitrary $x,y,z\in\mathbb{Z}$. Additivity imposes $g_ax+g_by-g_cz=0$, and the Wronskian is $a\,g_by-b\,g_ax$. Consider the matrix
\begin{equation}\label{eq:evaluation-matrix}
\begin{pmatrix}g_a&g_b&-g_c\\-b\,g_a&a\,g_b&0\end{pmatrix}.
\end{equation}
An invertible integer change of the coordinates $(x,y,z)$ sends its first row to $(\delta,0,0)$, with $\delta=\gcd(g_a,g_b,g_c)$, and its second row to some $(u,v,w)$. The kernel of the first row then consists of the vectors with first coordinate zero, on which the second row takes the values $\gcd(v,w)\mathbb{Z}$. The greatest common divisor of the two-by-two minors is invariant under the change of coordinates. It equals $\delta\gcd(v,w)$ after the change and $\gcd(c\,g_ag_b,b\,g_ag_c,a\,g_bg_c)$ before it, which proves~\eqref{eq:wronskian-image}.

Let $p$ be a prime dividing $a$, with $v_p(a)=\alpha$. Every term of $dn$ for $n=a$ has $p$-adic valuation at least $\alpha-1$ after applying $\psi$, so $v_p(d^{\psi}a)\ge\alpha-1$, while $v_p(a\,d^{\psi}b)\ge\alpha$. The first form in~\eqref{eq:int-wronskian} therefore gives $v_p(W^{\psi})\ge\alpha-1$. The same argument applies to the primes dividing $b$, and to those dividing $c$ through the second form in~\eqref{eq:int-wronskian}. Since $a$, $b$ and $c$ are pairwise coprime, $abc/R$ divides every Wronskian, and in particular $\Delta$. For the second claim, $1/p\le\log p/(2\log 2)$ for every prime $p$, with equality for $p=2$ alone, and summing $v_p(n)\log p$ gives $\log n$~\cite[Lemma~2.2]{Pasten2022}.
\end{proof}

\begin{theorem}[Pasten]\label{thm:pasten}
Let $(a,b,c)$ be a triple, and let $\psi\in\mathcal{T}(a,b)$ with $a$ and $b$ $\psi$-independent. Then
\[
\frac{c}{\log c}\le R\,\frac{\|\psi\|}{\log 2}.
\]
\end{theorem}

\begin{proof}
By Lemma~\ref{lem:int-divisibility} and $W^{\psi}\ne 0$, one has $abc/R\le|W^{\psi}|$. By~\eqref{eq:wronskian-sum} and the second part of Lemma~\ref{lem:int-divisibility}, $|W^{\psi}|\le ab\,\|\psi\|\log(ab)/(2\log 2)$. Dividing by $ab$ and using $\log(ab)\le 2\log c$ gives the claim~\cite[Theorem~3.3]{Pasten2022}.
\end{proof}

The norm $\|\psi\|$ thus takes the place of the degree drop in the proof of Theorem~\ref{thm:mason}. If $\|\psi\|$ were bounded, then the inequality would give the $abc$ conjecture up to a logarithm. In the quality plane, and with the exact divisibility and the prime weights kept, it reads as follows.

\begin{corollary}\label{cor:plane-form}
Under the hypotheses of Theorem~\ref{thm:pasten}, write $S(n)=\sum_{p\mid n}v_p(n)/p$ and put
\begin{equation}\label{eq:prime-weights}
\begin{gathered}
\sigma=\min\Bigl\{S(ab),\ \frac{c}{b}\,S(ac),\ \frac{c}{a}\,S(bc)\Bigr\},\qquad\nu=\frac{R\Delta}{abc},\\
\varepsilon_{\psi}=\frac{\log(\|\psi\|\,\sigma/\nu)}{\log R}.
\end{gathered}
\end{equation}
Then $\nu$ is a positive integer, $c\le R\,\|\psi\|\,\sigma/\nu$ and $X\le 1+\varepsilon_{\psi}$, and every presentation of the triple satisfies
\begin{equation}\label{eq:weighted-gain-bounds}
G_a\le 1+\varepsilon_{\psi}\qquad\text{and}\qquad G_p\le Y\le\lambda(1+\varepsilon_{\psi})<3(1+\varepsilon_{\psi}).
\end{equation}
\end{corollary}

\begin{proof}
Lemma~\ref{lem:int-divisibility} shows that $\nu$ is a positive integer and that the non-zero integer $W^{\psi}$ satisfies $|W^{\psi}|\ge\Delta=\nu\,abc/R$. Each of the three forms of $W^{\psi}$ in~\eqref{eq:int-wronskian} involves two of the integers $a$, $b$, $c$, say $x$ and $y$, and the computation leading to~\eqref{eq:wronskian-sum} gives $|W^{\psi}|\le xy\,\|\psi\|\,S(xy)$. The three forms therefore give $\nu c\le R\,\|\psi\|\,S(ab)$, $\nu b\le R\,\|\psi\|\,S(ac)$ and $\nu a\le R\,\|\psi\|\,S(bc)$. Multiplying the second inequality by $c/b$ and the third by $c/a$ gives three upper bounds for $\nu c$, whose minimum is $\nu c\le R\,\|\psi\|\,\sigma$, and taking logarithms gives the bound on $X$. The bounds on the gains follow from~\eqref{eq:dictionary} and $Y=\lambda X$ with $\lambda<3$. Since $\nu\ge 1$ and $\sigma\le S(ab)\le\log(ab)/(2\log 2)\le\log c/\log 2$ by Lemma~\ref{lem:int-divisibility}, the corollary contains Theorem~\ref{thm:pasten}.
\end{proof}

The prime weight $\sigma$ is the best coefficient for this estimate of the Wronskian, and the distance between the bound and the quality has two separate sources.

\begin{proposition}\label{prop:optimal-coefficient}
For every triple,
\begin{equation}\label{eq:optimal-wronskian-coefficient}
\max_{0\ne\phi\in\mathcal{T}(a,b)}\frac{|W^{\phi}|}{\|\phi\|}=ab\sigma.
\end{equation}
For $\psi$ as in Corollary~\ref{cor:plane-form}, put $m_{\psi}=|W^{\psi}|/\Delta$ and $\theta_{\psi}=|W^{\psi}|/(ab\,\|\psi\|\,\sigma)$. Then $m_{\psi}$ is a positive integer, $0<\theta_{\psi}\le 1$, and
\begin{equation}\label{eq:wronskian-gap}
1+\varepsilon_{\psi}-X=\frac{\log m_{\psi}+\log(1/\theta_{\psi})}{\log R}.
\end{equation}
\end{proposition}

\begin{proof}
Write $s_a=aS(a)$, $s_b=bS(b)$ and $s_c=cS(c)$, so that $|d^{\phi}n|\le\|\phi\|\,nS(n)$ for $n$ in $\{a,b,c\}$. Allow for a moment real values $\phi(\xi_p)$ of absolute value at most $1$. Since $a$, $b$ and $c$ have disjoint prime supports, the evaluations $u=d^{\phi}a$, $v=d^{\phi}b$ and $w=d^{\phi}c$ then range independently over $[-s_a,s_a]$, $[-s_b,s_b]$ and $[-s_c,s_c]$, subject to additivity $u+v=w$. The three forms in~\eqref{eq:int-wronskian} bound $|av-bu|$ by $bs_a+as_b$, $cs_a+as_c$ and $cs_b+bs_c$, whose minimum is $ab\sigma$.

This minimum is attained. If $|s_a-s_b|\le s_c$, then $(u,v)=(-s_a,s_b)$ attains the first bound. If $s_b>s_a+s_c$, then $(u,v)=(-s_a,s_a+s_c)$ attains the second. If $s_a>s_b+s_c$, then $(u,v)=(-s_b-s_c,s_b)$ attains the third. Since every value of $|av-bu|$ is at most each of the three bounds, the bound attained is their minimum. These evaluations come from rational prime values. For each $n>1$ among $a$, $b$, $c$, the prime values $\phi(\xi_p)=u_n/(nS(n))$ for $p\mid n$, with $u_a=u$, $u_b=v$ and $u_c=w$, give $d^{\phi}n=u_n$ and $|\phi(\xi_p)|\le 1$. Clearing denominators then gives a derivation $\phi\in\mathcal{T}(a,b)$ with $|W^{\phi}|/\|\phi\|\ge ab\sigma$. The reverse inequality is the estimate in the proof of Corollary~\ref{cor:plane-form}, which proves~\eqref{eq:optimal-wronskian-coefficient}.

For the second claim, Lemma~\ref{lem:int-divisibility} shows that $m_{\psi}$ is a positive integer, and~\eqref{eq:optimal-wronskian-coefficient} gives $\theta_{\psi}\le 1$. By the definition of $\nu$, one has $\nu c/R=\Delta/(ab)$, so that $R\,\|\psi\|\,\sigma/(\nu c)=ab\,\|\psi\|\,\sigma/\Delta=m_{\psi}/\theta_{\psi}$. Taking logarithms gives~\eqref{eq:wronskian-gap}.
\end{proof}

Pasten showed by Siegel's lemma~\cite{BombieriVaaler1983} that $\mathcal{T}(a,b)$ contains linearly independent derivations $\psi_1,\dots,\psi_r$, with $r=\omega(abc)-1$, such that $\prod_i\|\psi_i\|\le\omega(abc)\,c\log c/(2\log 2)$~\cite[Theorem~2.6]{Pasten2022}. The smallest of these norms is therefore at most $(\omega(abc)\,c\log c/(2\log 2))^{1/r}$, which gives small non-zero derivations~\cite[Corollary~2.7]{Pasten2022}. Independence is guaranteed for at least one index $i$, which need not be the index of smallest norm, and the norm of $\psi_i$ is then bounded by the whole product~\cite[Lemma~3.5]{Pasten2022}. Substituting the product bound into Theorem~\ref{thm:pasten} gives only $1\le R\,\omega(abc)(\log c)^2/(2(\log 2)^2)$, which places no restriction on the quality, whatever the value of $\omega(abc)$. Pasten therefore formulated the following conjecture, which we state for a given exponent.

\begin{conjecture}[Small derivatives conjecture with exponent $\eta$]\label{conj:sdc}
Let $0<\eta<1$. For all but finitely many triples whose entries do not form a set $\{1,N,q\}$ with $q$ prime, there is $\psi\in\mathcal{T}(a,b)$ with $a$ and $b$ $\psi$-independent and $\|\psi\|<c^{\eta}$.
\end{conjecture}

The small derivatives conjecture of Pasten~\cite[Conjecture~3.9]{Pasten2022} asserts that Conjecture~\ref{conj:sdc} holds for some exponent $\eta$.

\begin{proposition}\label{prop:sdc-gains}
Assume Conjecture~\ref{conj:sdc} with exponent $\eta$. Then, for every $\varepsilon>0$, all but finitely many triples satisfy $X<1/(1-\eta)+\varepsilon$, and every presentation of each of these triples satisfies
\begin{equation}\label{eq:uniform-sdc-gains}
G_a<\frac{1}{1-\eta}+\varepsilon\qquad\text{and}\qquad G_p<\frac{3}{1-\eta}+3\varepsilon.
\end{equation}
\end{proposition}

\begin{proof}
Consider a triple formed by $1$, $N$ and a prime $q$. Since $N=q\pm1$, either $q$ is odd and $N$ is even, or the triple is $1+2=3$. In both cases $R\ge 2q>c$, hence $X<1$~\cite[Lemma~4.1]{Pasten2022}. For the remaining triples outside the finite exceptional set of Conjecture~\ref{conj:sdc}, Theorem~\ref{thm:pasten} with $\|\psi\|<c^{\eta}$ gives $c^{1-\eta}<R\log c/\log 2$. Given $\varepsilon>0$, choose $0<\delta<1-\eta$ with $1/(1-\eta-\delta)<1/(1-\eta)+\varepsilon$. Once $c$ is large, one has $\log c/\log 2<c^{\delta}$, whence $c^{1-\eta-\delta}<R$ and $X<1/(1-\eta-\delta)$, and only finitely many triples have $c$ below a given bound. The bounds on the gains follow from~\eqref{eq:dictionary} and $Y<3X$.
\end{proof}

In this comparison, an $abc$ bound with exponent $M>1$ means the assertion
\begin{equation}\label{eq:abc-exponent-convention}
c<R^M\qquad\text{for all but finitely many triples.}
\end{equation}
This is a weak form of the $abc$ conjecture, with a fixed exponent. The finite exceptional set is necessary for $M$ close to $1$, since Reyssat's triple satisfies $c>R^{1.62}$. Pasten formulates the weak form for every triple, as the conjecture of Oesterl\'e~\cite[Conjecture~3.2]{Pasten2022}. His two implications remain valid under this convention, since Lemma~4.4 of~\cite{Pasten2022} uses $c<R^M$ only for the triple under consideration, and the proof of Theorem~4.5 of~\cite{Pasten2022} discards finitely many triples.

Conjecture~\ref{conj:sdc} with exponent $\eta$ gives~\eqref{eq:abc-exponent-convention} for every $M>1/(1-\eta)$ by Proposition~\ref{prop:sdc-gains}. Conversely, if~\eqref{eq:abc-exponent-convention} holds for some $1<M<2$, then Conjecture~\ref{conj:sdc} holds for every $\eta$ with $1-(2-M)/(4M)<\eta<1$~\cite[Theorem~4.5]{Pasten2022}. The $abc$ conjecture gives~\eqref{eq:abc-exponent-convention} for every $M>1$, and hence Conjecture~\ref{conj:sdc} for every $3/4<\eta<1$~\cite[Corollary~4.6]{Pasten2022}. In the other direction, Conjecture~\ref{conj:sdc} for every $0<\eta<1$ would give the $abc$ conjecture itself.

The two implications thus link the exponents without matching them, as Figure~\ref{fig:exponents} shows. An exponent $1<M<2$ leads to every $\eta>1-(2-M)/(4M)$, and from there back to every exponent above $4M/(2-M)$, which exceeds $4$. Starting from the $abc$ conjecture, the two results return the bounds~\eqref{eq:abc-exponent-convention} for exponents $M>4$ alone.

In the quality plane, the first implication is the bound $\limsup X\le 1/(1-\eta)$ as $c\to\infty$ of Proposition~\ref{prop:sdc-gains}. The gains enter only through~\eqref{eq:dictionary}. Pasten's inequality bounds the quality itself, and it bounds the approximation gain exactly for the presentation $P=R$, where $G_a=X$.

\begin{figure}[ht]
\centering
\begin{tikzpicture}
\begin{axis}[width=0.78\textwidth,height=0.46\textwidth,xmin=1,xmax=5,ymin=0,ymax=1,xlabel={$M$},ylabel={$\eta$},legend pos=south east,legend cell align=left,legend style={font=\small},tick label style={font=\small}]
\addplot[domain=1:2,samples=120,thick,lightblue] {1-(2-x)/(4*x)};
\addlegendentry{$\eta=1-(2-M)/(4M)$}
\addplot[domain=1:5,samples=120,thick,dashed,copper] {1-1/x};
\addlegendentry{$\eta=1-1/M$}
\draw[->,gray,thick] (axis cs:1.03,0.75) -- (axis cs:3.93,0.75);
\draw[fill=white] (axis cs:1,0.75) circle (1.8pt);
\draw[fill=white] (axis cs:4,0.75) circle (1.8pt);
\end{axis}
\end{tikzpicture}
\caption{The exponents in Pasten's two implications. An $abc$ bound with exponent $1<M<2$, in the sense of~\eqref{eq:abc-exponent-convention}, gives Conjecture~\ref{conj:sdc} for every $\eta$ above the solid curve, and Conjecture~\ref{conj:sdc} with exponent $\eta$ gives an $abc$ bound for every $M$ to the right of the dashed curve. The arrow leads from the $abc$ conjecture, which gives $\eta>3/4$, to the exponents $M>4$, and its open endpoints are limiting values.}\label{fig:exponents}
\end{figure}
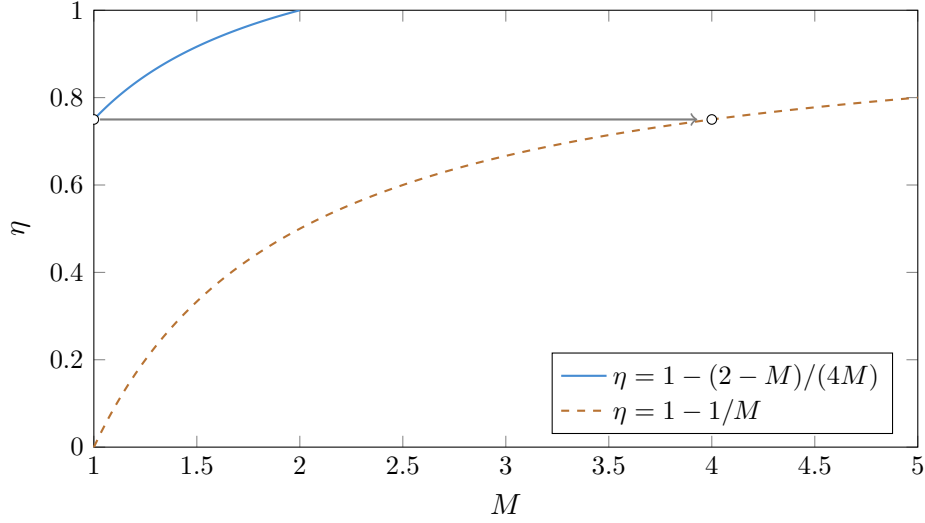

\section{Two measures of size}\label{sec:norms}

Pasten's norm measures a derivation by its values on the primes. Its values on the triple give another measure
\begin{equation}\label{eq:evaluated-size}
D_{\psi}=\max\{|d^{\psi}a|,|d^{\psi}b|,|d^{\psi}c|\}.
\end{equation}

\begin{proposition}\label{prop:value-norm}
Let $\psi\in\mathcal{T}(a,b)$ with $a$ and $b$ $\psi$-independent. Then
\begin{equation}\label{eq:value-norm-bounds}
\frac{\Delta}{c}=\frac{\nu ab}{R}\le D_{\psi}\le\|\psi\|\,\frac{c\log c}{2\log 2}.
\end{equation}
Moreover, $D_{\psi}>c/R$, and hence
\begin{equation}\label{eq:value-quality}
X<1+\frac{\log D_{\psi}}{\log R}.
\end{equation}
\end{proposition}

\begin{proof}
The upper bound follows from $d^{\psi}n=n\sum_{p\mid n}v_p(n)\psi(\xi_p)/p$ and Lemma~\ref{lem:int-divisibility}, applied to $n\le c$. For the lower bound, the first form in~\eqref{eq:int-wronskian} gives $|W^{\psi}|\le a|d^{\psi}b|+b|d^{\psi}a|\le cD_{\psi}$, and Lemma~\ref{lem:int-divisibility} gives $\Delta\le|W^{\psi}|$, whence $D_{\psi}\ge\Delta/c=\nu ab/R$. If $a\ge 2$, then $ab>c$ and $D_{\psi}>c/R$. If $a=1$, then $d^{\psi}1=0$, so that $d^{\psi}c=d^{\psi}b$ and $W^{\psi}=d^{\psi}b$, whence $D_{\psi}=|W^{\psi}|\ge\Delta\ge bc/R>c/R$. Taking logarithms gives~\eqref{eq:value-quality}.
\end{proof}

For every $\psi\in\mathcal{T}(a,b)$ with $W^{\psi}\ne 0$, the bound~\eqref{eq:value-quality} follows from additivity and the divisibility of the Wronskian, with no upper estimate for $\|\psi\|$, so that it gives no bound on the quality by itself. A bound on $X$ through this inequality would need a separate upper estimate for $D_{\psi}$, such as $D_{\psi}<R^{\delta}$, which would give $X<1+\delta$. Such an estimate is a strong hypothesis, since Example~\ref{ex:1-80-81} shows that even the smallest value of $D_{\psi}$ can exceed $c$, and that $\|\psi\|$ can grow without bound while $D_{\psi}$ stays fixed. For $a=1$ the smallest value of $D_{\psi}$ follows from the exact image of the Wronskian.

\begin{proposition}\label{prop:evaluation-minimum}
For a triple with $a=1$, one has $\Delta=\operatorname{lcm}\bigl(g(b),g(c)\bigr)$ and
\begin{equation}\label{eq:minimum-value-size}
\min\bigl\{D_{\psi}\colon\psi\in\mathcal{T}(1,b),\ W^{\psi}\ne 0\bigr\}=\Delta.
\end{equation}
If moreover $c=q^k$ for a prime $q$, then every $\psi\in\mathcal{T}(1,b)$ with $W^{\psi}\ne 0$ satisfies
\begin{equation}\label{eq:prime-power-obstruction}
\|\psi\|\ge\frac{g(b)}{\gcd\bigl(g(b),kq^{k-1}\bigr)}.
\end{equation}
\end{proposition}

\begin{proof}
For $a=1$ one has $g(a)=0$, and~\eqref{eq:wronskian-step} gives $\Delta=g(b)g(c)/\gcd(g(b),g(c))$. Since $d^{\psi}1=0$, additivity gives $d^{\psi}b=d^{\psi}c=W^{\psi}$, so that $D_{\psi}=|W^{\psi}|$, and~\eqref{eq:wronskian-image} gives the minimum. If $c=q^k$, then $d^{\psi}c=kq^{k-1}\psi(\xi_q)=g(c)\,\psi(\xi_q)$ is a non-zero multiple of $\Delta$, so that $\psi(\xi_q)$ is a non-zero multiple of $\Delta/g(c)=g(b)/\gcd(g(b),g(c))$.
\end{proof}

For $k=1$ the bound~\eqref{eq:prime-power-obstruction} reads $\|\psi\|\ge g(b)$, so that the triples $1+N=q$ excluded from Conjecture~\ref{conj:sdc} can carry large minimal norms. For $b=2^m$ and a prime $q=2^m+1$, for instance, one has $g(b)=m2^{m-1}$, of the order of $c\log c$, and Pasten observed the same size for $1+(2^n-1)=2^n$ with $2^n-1$ prime~\cite[Example~3.6]{Pasten2022}.

\begin{example}\label{ex:1-80-81}
For $1+80=81$, one has $g(80)=\gcd(160,16)=16$ and $g(81)=108$, so that Proposition~\ref{prop:evaluation-minimum} gives $\min D_{\psi}=432$, a value larger than $c=81$, and $\|\psi\|\ge 4$. The derivation with $(\psi(\xi_2),\psi(\xi_3),\psi(\xi_5))=(-3,-4,3)$ attains both minima. Adding $t(1,0,-10)$ with $t\in\mathbb{Z}$ leaves $d^{\psi}80$ and $d^{\psi}81$ unchanged, so that the norm grows without bound while $D_{\psi}$ stays equal to $432$. For $1+2400=2401$, one has $g(2400)=80$ and $g(2401)=1372$, which give $\min D_{\psi}=27440$ and $\|\psi\|\ge 20$, and the derivation $(9,-20,-11,20)$ on the primes $2$, $3$, $5$, $7$ attains both minima.
\end{example}

\section{Fermat triples and the presentation by bases}\label{sec:fermat}

Pasten deduced from Conjecture~\ref{conj:sdc} an asymptotic form of Fermat's last theorem by translating a Wronskian proof of the polynomial case~\cite[Propositions~3.10 and~3.11]{Pasten2022}. The presentation framework gives the same exponent bound through the approximation gain of a natural presentation.

\begin{proposition}\label{prop:fermat}
Let $a$, $b$, $c$ be pairwise coprime positive integers with $a^n+b^n=c^n$ for some $n\ge 2$, numbered so that $a<b$, which is possible since $a=b$ would force $c^n=2$. Consider the triple $(a^n,b^n,c^n)$ with the presentation $P=abc$ by its bases. Then $G_a>n/3$. In particular, if the quality of this triple is at most $\kappa$, then $n<3\kappa$.
\end{proposition}

\begin{proof}
The integer $abc$ divides $(abc)^n$ and has the same prime factors, so it is a presentation. Since $a<c$ and $b<c$, one has $\log(abc)<3\log c$, and therefore $G_a=n\log c/\log(abc)>n/3$. The last claim follows from $G_a\le X$ in~\eqref{eq:dictionary}.
\end{proof}

Under Conjecture~\ref{conj:sdc}, Propositions~\ref{prop:sdc-gains} and~\ref{prop:fermat} give
\begin{equation}\label{eq:fermat-threshold}
n\le\frac{3}{1-\eta}
\end{equation}
for all but finitely many Fermat triples. Indeed, choose $\varepsilon>0$ so that the interval $(3/(1-\eta),3/(1-\eta)+3\varepsilon)$ contains no integer. Outside a finite set one has $n<3G_a\le 3X<3/(1-\eta)+3\varepsilon$, and therefore $n\le 3/(1-\eta)$.

A single triple is a Fermat triple for finitely many exponents, since its largest entry is at least $2^n$. There is therefore an $n_0$ such that $a^n+b^n=c^n$ has no solution in coprime positive integers for $n\ge n_0$, which is Pasten's conditional form of the asymptotic Fermat theorem~\cite[Proposition~3.11]{Pasten2022}. The $abc$ conjecture gives $n\le 3$ in the same way, outside finitely many Fermat triples. A bound $G_a\le 3/2+\varepsilon$ for the presentation by bases, valid for every $\varepsilon>0$ outside a finite set, gives $n\le 4$, again outside finitely many Fermat triples.

\section{Numerical illustration}\label{sec:data}

For five triples, Table~\ref{tab:minimal} gives the smallest norm $\|\psi\|$ of a derivation $\psi\in\mathcal{T}(a,b)$ with $a$ and $b$ $\psi$-independent, together with one derivation attaining it, the prime weight $\sigma$ and the bound $1+\varepsilon_{\psi}$ of Corollary~\ref{cor:plane-form} for this derivation. Each displayed derivation satisfies additivity in exact integer arithmetic. The minima were found by exhaustive search, and the arguments after the table certify them. The factor $\nu$ equals $2$ for $1+80=81$ and $1$ for the other four triples. The values of $X$, $\sigma$ and the bound are rounded to four decimal places.

\begin{table}[ht]
\centering
\setlength{\tabcolsep}{3.5pt}
\begin{tabular}{llrrrlrr}
\toprule
Triple & $p\mid abc$ & $R$ & $X$ & $\min\|\psi\|$ & $\psi$ & $\sigma$ & $1+\varepsilon_{\psi}$\\
\midrule
$1+80=81$ & $2,3,5$ & $30$ & $1.2920$ & $4$ & $(-3,-4,3)$ & $1.3500$ & $1.2920$\\
$3+125=128$ & $2,3,5$ & $30$ & $1.4266$ & $6$ & $(1,-2,6)$ & $0.9333$ & $1.5065$\\
$1+2400=2401$ & $2,3,5,7$ & $210$ & $1.4557$ & $20$ & $(9,-20,-11,20)$ & $0.5717$ & $1.4557$\\
$1+4374=4375$ & $2,3,5,7$ & $210$ & $1.5679$ & $23$ & $(-23,-4,-23,-17)$ & $0.9431$ & $1.5754$\\
$2+3^{10}\cdot109=23^5$ & $2,3,23,109$ & $15042$ & $1.6299$ & $601$ & $(-601,38,586,79)$ & $0.7174$ & $1.6307$\\
\bottomrule
\end{tabular}

\medskip
\caption{Minimal derivations with $a$ and $b$ $\psi$-independent, and the bound on the quality with the divisibility factor $\nu$ and the prime weights $\sigma$ of~\eqref{eq:prime-weights}. The column $\psi$ gives one minimal derivation through its values $\psi(\xi_p)$ on the primes $p\mid abc$, listed in the same order.}\label{tab:minimal}
\end{table}

For the first and third rows, Proposition~\ref{prop:evaluation-minimum} gives the lower bounds $4$ and $20$ on the norm. For the fourth row, Corollary~\ref{cor:plane-form} gives $\|\psi\|\ge\nu c/(R\sigma)=243/11>22$. For the second row, write $x_p=\psi(\xi_p)$. Additivity reads $448x_2=x_3+75x_5$, and a norm at most $5$ forces $x_2=0$, since $448>5+75\cdot 5$, and then $x_3=x_5=0$.

For the last row, suppose that $\|\psi\|\le 600$ and $W^{\psi}\ne 0$. Lemma~\ref{lem:int-divisibility}, with $\Delta=23^4\cdot 19683$, and the second form in~\eqref{eq:int-wronskian} give
\begin{equation}\label{eq:reyssat-certificate}
W^{\psi}=23^4(10x_{23}-23x_2)\in 23^4\cdot 19683\,\mathbb{Z}\qquad\text{and}\qquad|10x_{23}-23x_2|\le 19800.
\end{equation}
After changing the sign of $\psi$ if necessary, one has $10x_{23}-23x_2=19683$. Reduction modulo $10$ gives $x_2\equiv 9\pmod{10}$, whence $x_2\ge-591$ and $10x_{23}\ge 19683-23\cdot 591=6090$, so that $x_{23}\ge 609$, a contradiction. The displayed derivation therefore has minimal norm $601$.

Every displayed derivation attains $|W^{\psi}|=\Delta$, with values $432$, $1600$, $27440$, $91125$ and $5508110403$. Thus $m_{\psi}=1$ in~\eqref{eq:wronskian-gap}, and every gap comes from the factor $\theta_{\psi}$ of the upper estimate for the Wronskian. The gaps are less than $0.080$, as Figure~\ref{fig:bound-gap} shows.

For $1+80=81$ and $1+2400=2401$, the minimum in~\eqref{eq:prime-weights} is $(c/b)S(ac)$ and $c=q^k$ is a prime power. The form $W^{\psi}=d^{\psi}c=kq^{k-1}\psi(\xi_q)$ then gives $\theta_{\psi}=|\psi(\xi_q)|/\|\psi\|=1$, and together with $m_{\psi}=1$ this gives equality in the quality bound, with $R\,\|\psi\|\,\sigma/\nu$ equal to $30\cdot 4\cdot(27/20)/2=81$ and $210\cdot 20\cdot(343/600)=2401$. For $1+4374=4375$, the corresponding form $d^{\psi}4375=3500\,\psi(\xi_5)+625\,\psi(\xi_7)$ involves two primes, and the minimal derivation has $|\psi(\xi_7)|=17<23=\|\psi\|$. The values $-23$ and $-17$ give $\theta_{\psi}=91125/(23\cdot 4125)=243/253$, whose contribution $\log(1/\theta_{\psi})/\log R$ is the gap of $0.0075$ in Table~\ref{tab:minimal}.

Seen from the side of the norm, the minimum equals the smallest value $\nu c/(R\sigma)$ that Corollary~\ref{cor:plane-form} allows for the first and third triples, and exceeds it by less than five percent for the last two. For Reyssat's triple, independence comes at a high price. The derivation with values $(20,3,46,0)$ on the primes $2$, $3$, $23$, $109$ gives $d^{\phi}a=10a$, $d^{\phi}b=10b$ and $d^{\phi}c=10c$, so that it lies in $\mathcal{T}(a,b)$ with $W^{\phi}=0$ and $\|\phi\|=46$.

No non-zero derivation in $\mathcal{T}(a,b)$ has smaller norm. The certificate~\eqref{eq:reyssat-certificate} shows that every $\phi\in\mathcal{T}(a,b)$ of norm below $601$ has $W^{\phi}=0$. For such $\phi$ the three quotients $d^{\phi}a/a$, $d^{\phi}b/b$ and $d^{\phi}c/c$ coincide, and the equations $10x_{23}=23x_2$ and $10x_3/3+x_{109}/109=x_2/2$ give
\begin{equation}\label{eq:reyssat-dependent-derivations}
(x_2,x_3,x_{23},x_{109})=(10t,3s,23t,545t-1090s)\qquad(s,t\in\mathbb{Z}).
\end{equation}
If the norm is below $46$, then $|t|\le 1$, and the last coordinate is either zero, with $s=t=0$, or at least $545$ in absolute value. The smallest non-zero norm is therefore $46$, attained with $(s,t)=(1,2)$, against $601$ under independence. For this triple, the product bound of~\cite[Theorem~2.6]{Pasten2022} guarantees a non-zero derivation of norm at most $(\omega(abc)\,c\log c/(2\log 2))^{1/3}\approx 663$. Corollary~2.7 of~\cite{Pasten2022} gives the asymptotic estimate $\|\phi\|<c^{1/3+\varepsilon}$ when $r=3$, for every $\varepsilon>0$ outside a finite exceptional set, and its leading scale $c^{1/3}\approx 186$ lies well above the true minimum $46$.

\begin{figure}[ht]
\centering
\begin{tikzpicture}
\begin{axis}[width=0.8\textwidth,height=0.42\textwidth,xmin=1.25,xmax=1.7,ymin=0.5,ymax=5.5,y dir=reverse,ytick={1,2,3,4,5},yticklabels={$1+80=81$,$3+125=128$,$1+2400=2401$,$1+4374=4375$,$2+3^{10}\cdot109=23^5$},xlabel={Quality $X$ and bound $1+\varepsilon_{\psi}$},legend pos=north east,legend cell align=left,legend style={font=\small},tick label style={font=\small}]
\draw[thick,gray] (axis cs:1.4266,2) -- (axis cs:1.5065,2);
\draw[thick,gray] (axis cs:1.5679,4) -- (axis cs:1.5754,4);
\draw[thick,gray] (axis cs:1.6299,5) -- (axis cs:1.6307,5);
\addplot[only marks,mark=*,mark size=2.2pt,lightblue] coordinates {(1.2920,1) (1.4266,2) (1.4557,3) (1.5679,4) (1.6299,5)};
\addlegendentry{$X$}
\addplot[only marks,mark=o,mark size=3.4pt,thick,copper] coordinates {(1.2920,1) (1.5065,2) (1.4557,3) (1.5754,4) (1.6307,5)};
\addlegendentry{$1+\varepsilon_{\psi}$}
\end{axis}
\end{tikzpicture}
\caption{The quality $X$ and the bound $1+\varepsilon_{\psi}$ of Corollary~\ref{cor:plane-form} for the minimal derivations of Table~\ref{tab:minimal}. For $1+80=81$ and $1+2400=2401$ the two values coincide.}\label{fig:bound-gap}
\end{figure}
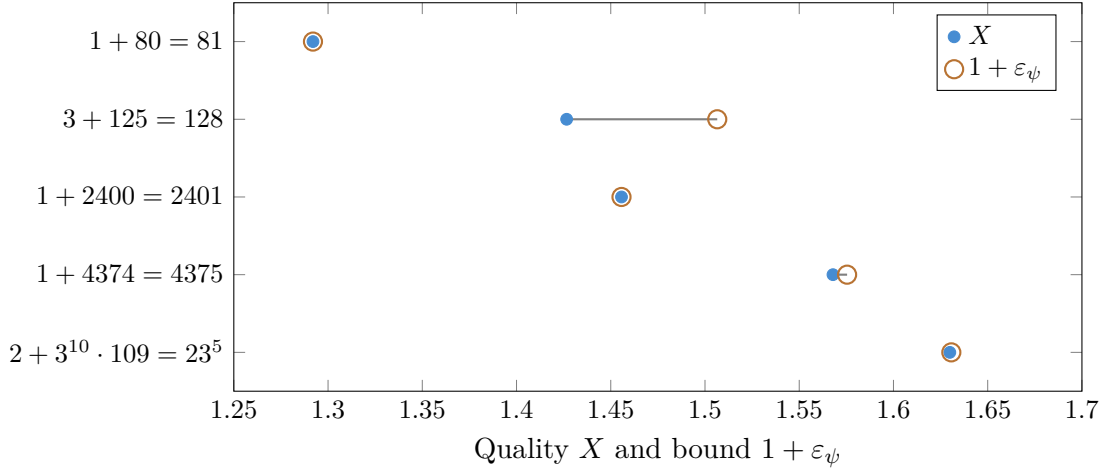

\FloatBarrier

\bibliographystyle{smfalpha}
\bibliography{pasten_gains}

\end{document}